\documentclass[a4paper, 12pt]{article}
\pdfoutput=1

\usepackage{tikz}
\usetikzlibrary{arrows, calc, patterns, positioning, shapes}
\usetikzlibrary{decorations.pathmorphing, backgrounds, scopes, fit}

\usepackage{hyperref}
\usepackage{theoremref}

\usepackage{sectsty}
\makeatletter
\def\@seccntformat#1{\csname the#1\endcsname. \textsc}
\makeatother
\sectionfont{\centering}

\usepackage[normalem]{ulem}
\usepackage{url}
\usepackage{dingbat}  

\newenvironment{keepintact}
  {\par\nobreak\vfil\penalty0\vfilneg
   \vtop\bgroup}
  {\par\xdef\tpd{\the\prevdepth}\egroup
   \prevdepth=\tpd}

\usepackage{amsmath, amssymb, mathtools, amsthm}
\usepackage{csquotes, datetime}
\usepackage[open,openlevel=1]{bookmark}

\usepackage[english]{babel}

\usepackage{microtype}
\usepackage{enumitem}
\usepackage[margin=1in]{geometry}
\hypersetup{
	bookmarks=true,
	unicode=true,
	colorlinks = true,
	urlcolor = blue,
	linkcolor = blue,
	citecolor = blue
} 

\newenvironment{enumrealm}{\setlength{\abovedisplayskip}{5pt}
\setlength{\belowdisplayskip}{5pt}}{\setlength{\abovedisplayskip}{10.0pt plus 2.0pt minus 5.0pt}
\setlength{\belowdisplayskip}{10.0pt plus 2.0pt minus 5.0pt}}

\usepackage[backend=bibtex, style=numeric, sorting=anyt,
url=false, sortcites=true]{biblatex}
\DeclarePairedDelimiter\abs{\lvert}{\rvert}

\newcommand{\ufd}{p_1^{\alpha_1} p_2^{\alpha_2} \cdots p_s^{\alpha_s}}
\newcommand{\ufdsh}{p_1^{\alpha_1} \cdots p_s^{\alpha_s}}

\newcommand{\aut}[1]{\operatorname{Aut}(#1)}
\newcommand{\cyc}[1]{\operatorname{C}_{#1}}

\newcommand{\Mod}[1]{\ (\mathrm{mod} \ #1)}

\newcommand{\qo}{\tilde{q}}
\newcommand{\pin}{\tilde{p}}
\newcommand{\tilpi}{\tilde{\pi}}
\newcommand{\piga}{\pi_\Gamma}
\newcommand{\pila}{\pi_\Lambda}
\newcommand{\fibo}[1]{F_{#1}}

\newcommand{\qlame}{\Lambda_{5,\text{I}}}
\newcommand{\qlamz}{\Lambda_{5,\text{II}}}
\newcommand{\qlamd}{\Lambda_{5,\text{III}}}
\newcommand{\slame}{\Lambda_{6, \text{I}}}
\newcommand{\slamz}{\Lambda_{6, \text{II}}}
\newcommand{\m}[1]{\text{M}_{#1}}
\newcommand{\hlame}{\Lambda_7}
\newcommand{\hthref}[1]{\hyperref[#1]{\thref{#1}}}

\NewDocumentCommand{\listspace}{}{}
\NewDocumentCommand{\textspace}{}{\linespread{1}\selectfont}

\NewDocumentCommand{\drawunspace}{}{\vspace{-0.5\baselineskip}}

\NewDocumentCommand{\case}{m}{\smash{\sffamily \textbf{Case #1.} \rmfamily }}
\NewDocumentCommand{\scase}{m}{\smash{\sffamily {Subcase #1.} \rmfamily }}

\NewDocumentCommand{\hold}{m}{Hölder}
\NewDocumentCommand{\hg}{m}{Hölder graph}
\NewDocumentCommand{\hgs}{m}{Hölder graphs}

\theoremstyle{plain}
\newtheorem{thm}{Theorem}[section]
\newtheorem{eufact}{Fact}[section]
\newtheorem{lem}{Lemma}[section]
\newtheorem{prop}{Proposition}[section]

\newtheorem{cor}[prop]{Corollary}

\theoremstyle{definition}

\setenumerate[1]{label=\alph*.}
\setenumerate[2]{label=\alph*.}

\tikzset
{%
point/.style={fill=black, draw, circle, inner sep=0.042cm, outer sep=0.12cm},
soint/.style={fill=black, draw, circle, inner sep=0.022cm, outer sep=0.1cm},
walk/.pic={
	\node [soint] (a) {};
	\node [soint,right=of a] (b) {};
	\path [->] (a) edge (b);
},
twalk/.pic={
	\node [soint] (a) {};
	\node [soint,right=of a] (b) {};
	\node [soint, right=of b] (c) {};
	\path [->] (a) edge (b)
	(b) edge (c);
},
ewalk/.pic={
	\node [soint] (a) {};
	\node [soint, right=of a] (b) {};
},
w/.style={baseline=-0.4ex},
ww/.style={baseline=-0.4ex, node distance=0.37cm}
}

\begin{document}
\title{Orders for which there exist exactly six or seven groups}
\date{May 6, 2024}
\author{Aban S. Mahmoud}
\maketitle

\begin{abstract}
	Much progress has been made on the problem of calculating $g(n)$ for various classes of integers $n$,  where $g$ is the group-counting function. We approach the inverse problem of solving the equations $g(n) = 6$ and $g(n) = 7$ in $n$. The determination of $n$ for which $g(n) = k$ has been carried out by G. A. Miller for $1 \le k \le 5$.
\end{abstract}

\textup{2020} \textit{Mathematics Subject Classification}: \textup{20D60, 05C69}.

\section{Introduction}
The function $g(n)$, defined as the number of groups of order $n$ (up to isomorphism), has various interesting arithmetical properties. It is an elementary fact, for example, that we have $g(p) = 1$ for all primes $p$, and an application of Sylow theorems shows that $g(pq) = 1$ whenever $p, q$ are primes with $q \not\equiv 1 \Mod{p}$ and $q > p$. This raises the question: Which $n$ satisfy $g(n) = 1$? Such $n$ are called \emph{cyclic} numbers. It turns out that $n$ is cyclic precisely when it is coprime with $\phi(n)$, where $\phi$ is Euler's totient function {\cite{szele}}. More generally, we can consider the equation $g(n) = k$ for a fixed integer $k$.

Miller answers this question for $k = 1, 2, 3$ in {\cite{miller1}} and for $k = 4, 5$ in {\cite{miller2}}. It has been re-derived multiple times: For instance, a short derivation of the cases $1 \le 3 \le 4$ appears in {\cite{olsson}}, and for $1 \le k \le 4$ in {\cite{gnumoas}}. In this paper, we consider the cases $k = 6$ and $k = 7$. After this paper was written, I was made aware of a 1936 paper {\cite{sigley}} addressing the case \mbox{$k = 6$.}

A beautiful formula for $g(n)$ when $n$ is square-free has been discovered by \hold{1} {\cite[Thm.~5.1]{gnumoas}}. An application of it is that, if $p, q, r$ are primes with $q \equiv r \equiv 1 \Mod{p}$ and $qr$ is a cyclic number, then $g(pqr) = p + 2$, already showing that the image of $g$ is infinite. There is no similarly explicit formula for cube-free $n$, but there is an efficient algorithm {\cite{cube-free}}, and there are relatively simple formulae for $g(n)$ when $n$ has a small number of prime factors with small exponents.

By forming direct products, it is clear that $g(ab) \ge g(a)g(b)$, at least when $a$ and $b$ are coprime. Since \mbox{$g(p^2) = 2$,} $g(p^3) = 5$ and $g(p^\alpha) \ge 14$ for all primes $p$ and all $\alpha \ge 4$ (see, for example, \cite[Thm.~3.1]{gnumoas}), it is enough to consider fourth-power-free $n$. The key computational tool will be directed graphs: We associate a directed graph to each number, with the vertices being the prime powers that divide $n$ and the edges indicating the existence of nontrivial semidirect products.

\section{Hölder graphs}
\subsection{Hölder's formula}
Given a square-free integer $n$, a prime factor $p$ and a subset $\pi$ of its prime factors, we let $v(p, \pi)$ be the number of $q$ in $\pi$ such that $q \equiv 1 \Mod{p}$. In this case, we say that $p$ is \textbf{related} to $q$, and $p$ and $q$ are related. We also let $\Gamma$ be the set of all prime factors of $n$. \mbox{Then we have}
\begin{thm}[\textbf{\hold{1}'s formula}]\thlabel{euholder} In this notation,
	\begin{equation*}
		g(n) = \sum_{\pi \subseteq \Gamma} \prod_{p \notin \pi} \frac{p^{v(p, \pi)} - 1}{p - 1}
	\end{equation*}
\end{thm}
\begin{proof} See {\cite[Thm.~5.1]{gnumoas}}. \end{proof}

It is therefore natural to define the \emph{\hg{1}} $\Gamma = \Gamma(n)$ of a square-free integer $n$ to be the unlabeled directed graph whose vertices are the prime factors, and where we have an edge $p \rightarrow q$ precisely when $p$ is related to $q$. Observe also that when the out-degree of every vertex is at most 1, $g(n)$ depends solely on the shape of the \hg{1}, and therefore we are justified in speaking of $g(\Gamma).$ We call both $n$ and $\Gamma$ \emph{regular} in this case. We will occasionally treat directed acyclic graphs $\Gamma$ where we only know the \emph{labels}, or the numerical values, of the vertices whose out-degree exceeds $1$. In this case, we can still speak of $g(\Gamma)$ without ambiguity, as it will be independent of the labels of the rest of the vertices. We also observe that Dirichlet's theorem on arithmetic progressions shows that a given (unlabeled) directed acyclic graph is always realizable as the \hg{1} of some odd square-free integer.

We generalize this to arbitrary $n = \ufdsh$ as follows: The vertices are now the $p_i^{\alpha_i}$, and an edge $p^\alpha \rightarrow q^\beta$ exists precisely when $q^j \equiv 1 \Mod{p^i}$ for some $1 \le i \le \alpha$ and $1 \le j \le \beta$. We will call this the \emph{generalized \hg} of $n$, also denoted by $\Gamma(n)$. Typically, $n$ will be cube-free, in which case we write $p^2 \dashrightarrow q$ to mean $p \parallel q - 1$, and $q \dashrightarrow p^2$ to mean $q \mid p + 1$ but $q \nmid p - 1$. We call these two types of arrows \emph{weak}, and other arrows \emph{strong}.

A vertex is said to be \emph{initial} if it has an out-degree of $0$, and \emph{terminal} if it has an in-degree of $0$. We also let $S(\pi, \Gamma) = \prod_{p \notin \pi}h_\Gamma(p, \pi)$ be the summand in Hölder's formula corresponding to the subset $\pi$. If there is a vertex $p$ not connected to any vertex in $\pi,$ the entire summand $S(\pi \Gamma)$ vanishes. Motivated by this observation, we call a subset $\pi$ \emph{central} if, for every vertex outside of it, there exists at least one edge to a vertex inside it -- equivalently, if $S(\pi \Gamma)$ does not vanish. We remark that $n$ is regular precisely when $g(\Gamma(n))$ is equal to the number of central subsets.

\subsection{Connectivity and multiplicativity}
We have already observed that $g(ab) \ge g(a)g(b)$ for coprime $a, b$. When does equality hold? If it does hold, then every group of order $ab$ splits as a direct product of groups of orders $a$ and $b.$ Coprimality is therefore a necessary condition: If $p$ divides both $a$ and $b,$ we simply observe that the group $\cyc{a/p} \times \cyc{p^2} \times \cyc{b/p}$ does not split in this way. Next, it must be impossible to form semidirect products. If $p$ divides $b$ and $q^n$ divides $a$, we observe that $\abs{\aut{\cyc{q}^n}}$ has $q^i - 1$ as a factor for all $1 \le i \le n$, and therefore a nontrivial semidirect product exists if there is a relation $p \rightarrow q^n.$ In this case, $a$ and $b$ are said to be \emph{arithmetically dependent}. We conclude that another necessary condition is $\Gamma(n) = \Gamma(a) \sqcup \Gamma(b)$.  The converse is also true:

\begin{thm}\thlabel{eufrobenius}
	Suppose we have $n = ab$. Then the following are equivalent:
	\begin{enumerate}
		\item $g(n) = g(a)g(b)$,
		\item $\abs{G} = n$ implies $G \cong A \times B$ where $\abs{A} = a$ and $\abs{B} = b$,
		\item $a$ and $b$ are arithmetically independent,
		\item $\Gamma(n) = \Gamma(a) \sqcup \Gamma(b)$.
	\end{enumerate}
\end{thm}
\begin{proof}
	See {\cite[Lem.~21.19]{monolith}}.
\end{proof}

\begin{prop}\thlabel{euufd}
	For positive $n,$ we have a unique factorization $n = n_1 \cdots n_k m$ where\pagebreak[3]
	\begin{enumerate} \listspace
		\item The $k + 1$ factors are pairwise arithmetically independent,
		\item Each $n_i$ is connected with $g(n_i) \ge 2$,
		\item $m$ is cyclic, that is, $g(m) = 1$.
	\end{enumerate} \textspace
\end{prop}
\begin{proof}
	This follows immediately by considering the generalized \hg{1} of $n$ and letting the $n_i$ be the factors corresponding to the connected component with at least two vertices, and $m$ the product of the values of the isolated vertices.
\end{proof}

In solving $g(n) = c$, it is useful to define the \emph{cyclic part} of $n$ to be $m$ (in this notation), and call $n$ \emph{cyclic-free} if $m = 1$. Next, whenever $c = c_1 \cdots c_s$, finding solutions $g(n_i) = c_i$ with pairwise independent $n_i$ leads to a solution to $c$. It is therefore natural to consider ``prime'' values of $n$, that is, those with $k = 1$. We will call cyclic-free $n$ with $k = 1$ \emph{connected}. By considering non-nilpotent groups, we can get a lower bound on $g(n)$. We will only make use of the following (weak) lower bound:

\begin{prop}\thlabel{eunnp}
	If $n = \ufdsh$ is connected, then $g(n) \ge g(p_1^{\alpha_1})\cdots g(p_s^{\alpha_s}) + s - 1.$
\end{prop}

\subsection{Splicing graphs}
Given disjoint graphs $\Gamma$ and $\Lambda$ with a fixed terminal vertex $\qo$ in $\Gamma$ and any vertex $\pin$ in $\Lambda$, we can make a new graph $\Gamma \rightarrow \Lambda$ by adjoining an arrow $\pin \rightarrow \qo$ to the union of $\Gamma$ and $\Lambda$. This depends on the choice of $\qo$ and $\pin$ in general, but we will not explicitly refer to them when there is no risk of confusion. In addition, we tacitly assume that we are given the functions $h_\Gamma$ and $h_\Lambda$. We use the notation $g(X; v)$ to refer to the sum in Hölder's formula but only considering subsets $\pi$ with $v \in \pi$.

\begin{prop}\thlabel{eujoin}
	In this notation, \begin{enumrealm}\begin{equation*} g(\Gamma \rightarrow \Lambda) = g(\Gamma)g(\Lambda) + g(\Gamma - \{\qo\})g(\Lambda; \pin). \end{equation*}\end{enumrealm}
\end{prop}
\begin{proof}
	Let $M = \Gamma \rightarrow \Lambda.$ Given a subset $\tilpi$ of $M$, we define $\piga = \pi \cap \Gamma$ and $\pila = \pi \cap \Lambda$. For all $p \neq \qo$, we have $h_M(p, \tilpi) = h_\Gamma(p, \piga)$ if $p$ is in $\Gamma$, and $h_\Lambda(p, \pila)$ otherwise. We have three cases:

	If $\qo \in \piga$, we have $h_M(q, \piga) = h_\Gamma(q, \piga)$ and $h_M(p, \pila) = h_\Lambda(p, \pila)$ for all $q \in \Gamma - \piga$ and $p \in \Lambda - \pila$. It follows that $S(\tilpi, M) = S(\piga, \Gamma)S(\pila, \Lambda).$
	
	On the other hand, if $\qo\notin\piga$ and $\pin \in \pila$, we get $h_M(\qo, \tilpi) = 1$ because there is a unique edge $\qo \rightarrow \pin,$ so it contributes nothing to the product. Since $h_M(p, \tilpi) = h_\Lambda(p, \pila)$ for all $p \in \Lambda - \pila$, we have $S(\tilpi, M) = S(\piga, \Gamma - \{\qo\})S(\pila, \Lambda)$.
	
	Finally, if $\qo \notin \piga$ and $\pin\notin\pila,$ the previous analysis shows that $h_M(\qo, \tilpi) = 0$, which in turn implies that $S(\tilpi, M) = 0$, and thus $\tilpi$ contributes nothing to the sum.
	
	The subsets $\tilpi \subseteq M$ correspond bijectively to pairs $(\piga, \pila)$ with $\piga \subseteq \Gamma$ and $\pila \subseteq \Lambda.$ From the pairs with $\qo \in \piga$ we get $g(\Gamma)g(\Lambda)$. From the pairs with $\qo \notin \piga$, we need $\pin \in \pila,$ and the resulting terms combine to give precisely $g(\Gamma - \{\qo\})g(\Lambda; \pin).$
\end{proof}

\begin{cor}\thlabel{eustick} Suppose $\Gamma$ is a \hg{1} and $v \notin \Gamma$.\listspace
	\begin{enumerate} \listspace
		\item Fix a terminal vertex $q$ in $\Gamma$. Then $g(\Gamma \rightarrow v) = g(\Gamma) + g(\Gamma - \{q\}).$
		\item Fix an initial vertex $p$ in $\Gamma.$ Then $g(v \rightarrow \Gamma) = g(\Gamma) + g(\Gamma; p).$
	\end{enumerate}\textspace
\end{cor}

We can iterate the first operation, starting with a graph $\Gamma_0$ and a fixed terminal vertex $q$ in it and define $\Gamma_{n + 1} = \Gamma_{n} \rightarrow v_{n + 1}$, where the $v_i$ are all distinct. If we let $\alpha = g(\Gamma_0 - \{q\})$ and $\beta = g(\Gamma_0)$, then the sequence $a_n = g(\Gamma_n)$ satisfies the recurrence relation $a_{n + 2} = a_{n + 1} + a_{n}$ with initial values $a_{-1} = \alpha$ and $a_0 = \beta$. Starting with a single vertex, we get

\begin{prop}\thlabel{eufibo}
	Let $\Phi_n$ be a directed path of $n$ vertices. Then $g(\Phi_n) = \fibo{n + 1}$, where $\fibo{n}$ is the Fibonacci sequence.
\end{prop}

An amusing consequence of this fact is this. Since $\Phi_{m + n} = \Phi_m \rightarrow \Phi_n$, and we have just seen that $g(\Phi_m - \{\text{final}\}) = \fibo{m}$ and $g(\Phi_n; \text{initial}) = \fibo{n}$, 
we can now apply \hyperref[eustick]{\hthref{eustick}} to obtain the well-known identity, \begin{enumrealm}
\begin{align*}
	\fibo{m + n} &= g(\Phi_m \rightarrow \Phi_{n - 1}) \\
	&= g(\Phi_m)g(\Phi_{n - 1}) + g(\Phi_m - \{\text{final}\})g(\Phi_{n - 1}; \{\text{initial}\}) \\
	&= \fibo{m + 1}\fibo{n} + \fibo{m}\fibo{n - 1}.
\end{align*}\end{enumrealm}
\vspace{-\baselineskip}
\subsection{Some lower bounds}
We have already seen that $g(\Phi_n)$ is the Fibonacci sequence, and so it grows exponentially with $n$. It follows that $g$ grows exponentially in the length of the longest directed path, since $g(\Gamma) \ge  g(\Gamma_0)$ for all subgraphs $\Gamma_0 \subseteq \Gamma$. It grows exponentially in the in- and out-degrees as well:
\begin{lem}\thlabel{euinout}
	Let $r, s$ denote the in- and out-degree, respectively, of a vertex $v$ in a \hg{1} $\Gamma$ with label $p$. Then $g(\Gamma) \ge 2^r + \frac{p^s - 1}{p - 1}$.
\end{lem}
\begin{proof}
	Suppose we have $\alpha_i \rightarrow v$ and $v \rightarrow \beta_j$ for $1 \le i \le r$ and $1 \le j \le s$, and let $\Gamma_0$ be the subgraph induced by the collection of vertices $\alpha_i$ and $\beta_j$. Evidently, any subset of $\Gamma_0$ containing $v$ and the $\beta_j$ is central, and there are $2^r$ such subsets. Moreover, the complement of $\{v\}$ in $\Gamma_0$ is central, and it contributes $\frac{p^s - 1}{p - 1}$. This shows that $\Gamma_0$, and thus $\Gamma$, satisfies the inequality.
\end{proof}	

\section{Solving $g(n) = 6, 7$}
We call $n$ such that $g(n) = 6\text{ or }7$ \emph{admissible}. The analysis will be split into two parts, depending on the connectivity of $n$. From section 3.2 onward, we assume that $n$ is connected. We use the standard notation $\lambda(n) = \sum \alpha_i$ for $n = \ufd$. Under the assumption of connectivity, \hthref{euinout} restricts the possible factorizations of admissible $n$, once one recalls that $g(p^3) = 5$ and $g(p^2) = 2$.

\begin{lem}\thlabel{euclass}
	Suppose $n$ is admissible and connected. Then either $n$ is square-free, or $n$ is a divisor of a number of one of the forms $p^2 q^2 r, p^2 q r s, p^3 qr$. In particular, $\lambda(n) \le 5$.
\end{lem}
\begin{proof}
	The only part that needs proof is that $g(n) \ge 8$ for $n = p^2 n_0$ where $n_0 = q_1 q_2 q_3 q_4$ is squarefree. If not, then we need $g(n_0) \le 3$ by the connectivity of $n$. First, if $g(n_0) = 3$ then it can be seen that $n_0$ must have an isolated vertex $q$, which must therefore be connected with $p^2$, yielding $g(p^2 q)g(n_0 / q) + 1 \ge 10$ groups. Next, if $g(n_0) = 2$, it follows that $n_0$ has two isolated vertices $q_1$ and $q_2$. But then $g(p^2 q_1 q_2) \ge 4$ by \hthref{eunnp}, and consequently $g(n) \ge 2g(p^2 q_1 q_2) \ge 8$. Finally, $g(n_0) = 1$ means that all the vertices are connected with $p^2$, which is ruled out by \hthref{euinout}.
\end{proof}
\subsection{Disconnected numbers}
We first dispose of the disconnected case. Let $n = n_1 \cdots n_k m$ be the unique factorization of $n$ in the notation of \hthref{euufd}, with $m = 1$ and $n_i$ ordered such that $g(n_i)$ is non-decreasing. The \emph{primality} of $7$ shows that $k = 1$ when $g(n) = 7$, so it is enough to consider connected $n$ in this case. If $g(n) = 6$, either $k = 1$ or $g(n_1) = 2$ and $g(n_2) = 3$. For the sake of completeness, we briefly re-derive the solutions in $n_1$ and $n_2$ here.

First, $n_1$ is cube-free and has at most one square factor. If it has none, then it must have the graph \tikz[ww] \pic {walk};. If it has one, then $n_1 = p^2$. Next, if $n_2$ is square-free, \hthref{euinout} forces its graph to be \tikz[ww] \pic {twalk};. If it is not, then it is easy to see that it must be of the form $p^2 q$ with $q \dashrightarrow p^2$ and $p, q$ odd. \hthref{euppq} asserts that this is also a sufficient condition. We summarize the results:

\begin{prop}
	Suppose $n$ is cyclic free.
	\begin{enumerate}
		\item $g(n) = 2$ if and only if either $n = p^2$ or $n = pq$ where $q \equiv 1 \Mod{p}$.
		\item $g(n) = 3$ if and only if $n$ is odd, and either $n = p^2 q$ and $q \mid p + 1$, or $n = pqr$ where $q \equiv 1 \Mod{p}$ and $r \equiv 1 \Mod{q}$ but $r \not\equiv 1 \Mod{p}$.
	\end{enumerate}
\end{prop}

\subsection{Squarefree admissible numbers}
Let $i(\Gamma), o(\Gamma)$ denote the maximum in- and out-degree of a vertex in a \hg{1} $\Gamma$. If $o(\Gamma) = 3$, we get at least $1 + \frac{p^3 - 1}{p - 1} = p^2 + p + 2 \ge 8$. It follows that $o(\Gamma) \le 2$, and, for a similar reason, $i(\Gamma) \le 2$. We therefore have four cases to consider.

\case{1} $i(\Gamma) = 1, o(\Gamma) = 1$. In this case $\Gamma$ is simply a path $\Phi_k$ for some $k$, and as neither 6 nor 7 is a Fibonacci number, so by \hthref{eufibo} this is never an admissible graph.

\case{2} $i(\Gamma) = 2, o(\Gamma) = 1$. It is clear that there is a unique vertex $v$ with in-degree 2, so let $w_1, w_2 \rightarrow v$ be arrows, and call this subgraph $K$. We have $g(K) = 2^2 = 4$, and the only way to add edges is by adding vertices, since we have $o(\Gamma) = 1$. A new vertex $u$ can be connected to $w_1$ to get $u \rightarrow w_1$, this would give $g(K) + g(K; w_1) = 4 + 2 = 6$ groups by \hthref{eujoin}, so call this new graph $Q$. We cannot extend $Q$: Both extensions $b \rightarrow Q$ and $Q \rightarrow f$ give too many groups by \hthref{eustick}. In the other direction, we could add $v \rightarrow f_1$ in $K$ instead, yielding $g(K) + g(\tikz[ww] \pic{ewalk};) = 5$. However, this can only be extended forward, that is, by taking $(K \rightarrow f_1) \rightarrow f_2$, and this yields $g(K \rightarrow f_1) + g(K) = 4 + 5 = 9$ groups. Thus the only admissible graph is $Q$, with $g(Q) = 6$.

\case{3} $i(\Gamma) = 1, o(\Gamma) = 2$. It is clear that there is a unique vertex with out-degree 2, so label it $p$ for some prime $p$. We get at least $p + 2$ groups, so we have $p \le 5$ and therefore $p$ is initial (there can be no vertex labeled 2 because it would have an out-degree of at least 2). We cannot add edges without adding vertices. Using \hthref{eustick}, we see that adding a new vertex $v$ and an arrow $u \rightarrow v$, where $u$ is either of the two vertices connected with $p$, yields $p + 4$ groups. As $p > 2$, we must have $p = 3$. This completes the analysis for all such graphs: We get admissible graphs when $p = 5$ in the first sub-case and when $p = 3$ in the second sub-case, both yielding $g = 7$.

\case{4} $i(\Gamma) = 2, o(\Gamma) = 2$. 
A vertex labeled $p$ with in-degree and out-degree both equal to $2$ has $p > 2$, and would yield $4 + p + 1 = p + 5 > 7$ groups. So let $u$ be the unique vertex with out-degree $2$ and label it $p$, and assume we have $u \rightarrow v_1, v_2$, and call this graph $Q$. We have already seen that $g(Q) = p + 2$. We show that we cannot add a vertex $w$. First, we cannot connect it with $u$: An arrow $w \rightarrow u$ is impossible since it means $w = 2$, and $u \rightarrow w$ is ruled out by the condition on $o(\Gamma)$. Adding $w \rightarrow v_1$ does not give an admissible graph, because the subsets $\{v_1, v_2\}$ and $\{w, v_1, v_2\}$ would both be central, yielding $2(p + 1) = 2p + 2 \ge 8$ groups in total. Finally, if we add $v_1 \rightarrow w$ instead, we are back to Case $3$ with $p + 4 \ge 7$ groups, and it is now impossible to achieve $i(\Gamma) = 2$.

We can add an edge $v_1 \rightarrow v_2$ to $Q$ without adding $w$. This will give $p + 4$ groups, and is the only admissible case, giving $6$ if $p = 2$ and $7$ if $p = 3$.

\subsection{Small non-square-free admissible numbers}
In this section, we consider non-square-free admissible $n$ with $\lambda(n) \le 4$. Following {\cite{bettinafour1}}, we list special cases of formulae for $g(p^2 q^2), g(p^2 q r), g(p^2 q) \text{ and } g(p^3 q)$ and analyze them one by one. In their notation, $w_r(s)$ is defined to be 1 if $s \mid r$ and 0 otherwise. 

\begin{eufact}\thlabel{euppq}
	For all odd primes $p, q$, we have $g(2p^2) = 5$ and \[g(p^2 q) = 2 + \frac{q + 5}{2} w_{p - 1}(q) + w_{p + 1}(q) + 2w_{q - 1}(p) + w_{q - 1}(p^2).\] In particular, $g(p^2 q) = 3$ if and only if $q \mid p + 1$, and $g(p^2 q) = 4 + \frac{q + 1}{2}$ if $q \mid p - 1$.
\end{eufact}

\textit{Analysis.} We have $p \nmid q - 1$, since otherwise we would get at most $5$ groups. Thus we need $q \mid p - 1$, and this yields admissible $n$ precisely when $q = 3$, giving $6$ groups, and when $q = 5$, giving $7$ groups.

\begin{eufact}\thlabel{eupppq}
	If $p^3 q$ is admissible, then $p$ and $q$ are odd and
	$$\begin{aligned}
		g(p^3 q) = 5 &+ \frac{q^2 + 13q + 36}{6} w_{p - 1}(q) + (p + 5) w_{q - 1}(p) \\
		&+ \frac{2}{3} w_{q - 1}(3)w_{p - 1}(q) + w_{(p + 1)(p^2 + p + 1)}(q) (1 - w_{p - 1}(q)) \\
		&+ w_{p + 1}(q) + 2 w_{q - 1}(p^2) + w_{q - 1}(p^3).
	\end{aligned}$$
\end{eufact}

\textit{Analysis.} If $p \mid q - 1$, we have at least $p + 10$ groups. It is also clear that we cannot have $q \mid p - 1$. We therefore have $g(p^3 q) = 5 + w_{p + 1}(q) + w_{(p + 1)(p^2 + p + 1)}(q)$. The last summand is $1$, and thus $g(p^3 q) = 6$ if $q \mid (p^2 + p + 1)$ and $g(p^3 q) = 7$ if $q \mid p + 1$.

\begin{eufact}\thlabel{euppqq}
	If $p < q$ and $p^2 q^2$ is admissible, then $p$ and $q$ are odd and
	\[g(p^2 q^2) = 4 + \frac{p^2 + p + 4}{2} w_{q - 1}(p^2) + (p + 6)w_{q - 1}(p) + 2w_{q + 1}(p) + w_{q + 1}(p^2).\]
\end{eufact}

\textit{Analysis.} It is clear that we need $p \nmid q - 1$, so $p \mid q + 1$. This gives 6 or 7 groups, depending on whether $p^2 \mid q + 1$ as well or not.

\begin{eufact} \thlabel{euppqr}
	If $n = p^2 q r$ is admissible with $q < r$, then $q > 2$ and $g(n) = h(n) + k(n)$ where$$\begin{aligned}
		h(n) &= 2 + w_{p^2 - 1}(qr) + 2w_{r - 1}(pq) + w_{r - 1}(p)w_{p - 1}(q) + w_{r - 1}(p^2 q) \\ 
		&+ w_{r - 1}(p)w_{q - 1}(p) + 2w_{q - 1}(p) + 3w_{p - 1}(q) + 2w_{r - 1}(p) \\ 
		&+ 2w_{r - 1}(q) + w_{r - 1}(p^2) + w_{q - 1}(p^2) + w_{p + 1}(r) + w_{p + 1}(q), \\
		k(n) &= \frac{qr + 1}{2} w_{p - 1}(qr) + \frac{r + 5}{2} w_{p - 1}(r)(1 + w_{p - 1}(q))\\
		&+ (p^2 - p)w_{q - 1}(p^2)w_{r - 1}(p^2) \\
		&+ (p - 1)(w_{q - 1}(p^2)w_{r - 1}(p) + w_{r - 1}(p^2)w_{q - 1}(p) + 2w_{r - 1}(p)w_{q - 1}(p)) \\
		&+ \frac{(q - 1)(q + 4)}{2} w_{p - 1}(q)w_{r - 1}(q) \\
		&+ \frac{q - 1}{2} (w_{p + 1}(q)w_{r - 1}(q) + w_{p - 1}(q) + w_{p - 1}(qr) + 2w_{r - 1}(pq)w_{p - 1}(q)).
	\end{aligned}$$
\end{eufact}

\textit{Analysis.} This case is more complicated, so we divide the proof into several steps. We begin by applying the formula to the graph $D(p)$ with $p^2 \dashrightarrow q, r$: It can be seen that $g(D(p)) = 2p + 5 \ge 9$.

We see that $p \neq 2$, since otherwise we would have $D(2)$ as a subgraph. Next, we show that there is no strong arrow to $p^2$. By \hthref{euppq} we see that there is no $r \rightarrow p^2$ since $r \ge 5$. If we have $q \rightarrow p^2$, we get $g(p^2 q) = 6$, and the connectivity of $n$ forces one of $w_{r - 1}(p), w_{r - 1}(q)$ and $w_{p + 1}(r)$ to be equal to $1$. The formula then implies that we get
$2(p - 1)$, $\frac{(q - 1)(q + 4)}{2}$, \text{or} $w_{p^2 - 1}(qr) + w_{p + 1}(r)$
more groups, respectively, showing that $n$ is inadmissible. 

Observe that $h$ depends only on the shape of the generalized graph of $n$, and not on the magnitudes of its prime factors. Consequently, if $k(n) = 0$, we can calculate $g(n)$ for all $n$ with a specific graph simply by calculating $g(n)$ for a single representative, using the Cubefree package in GAP{\cite{GAP4, cubefreepkg}}. We will call $n$ ``regular'' in this case as well.

We analyze $k(n)$. If there is an arrow $q \rightarrow r$, then we have $h(n) \ge 2 + 2w_{r - 1}(q) = 4$, and if there is no arrow then both $q$ and $r$ are connected with $p^2$. A case-by-case analysis shows that $h(n) \ge 5$ in this case. Consequently, we must have $k(n) \le 3$ in all admissible cases, with equality implying the existence of an arrow $q \rightarrow r$.

Assume now that $k(n) > 0$, that is, $n$ is not regular. By considering the ``coefficients'' of the $w$-sums in the formula for $k$, we observe that all of them are greater than $3$ except possibly $p - 1$ and $\frac{q - 1}{2}$. First, the sum whose coefficient is $(p - 1)$ must vanish, as otherwise would get $D(p)$ as a subgraph.

Secondly, we consider $\frac{q - 1}{2}$. Our analysis shows that all the summands but $w_{p + 1}(q)w_{r - 1}(q)$ must vanish. Let $\m1(q)$ be the graph with $q \rightarrow r$ and $q \dashrightarrow p^2$ (we emphasize its dependence on $q$). We see that $g(\m1(q)) = 5 + \frac{q - 1}{2}$, thus we get admissible $n$ precisely when $q = 3$ or $q = 5$, yielding $6$ and $7$ groups, respectively, and this is the only irregular admissible case.

\subsubsection*{Regular \hgs{1} with two edges}
Having dealt with the ``irregular'' cases, we consider regular $n$ with two edges in their graphs, and calculate $g$ for a representative using GAP. It is convenient to list $\m1(q)$ as well.
\vspace*{-\baselineskip}
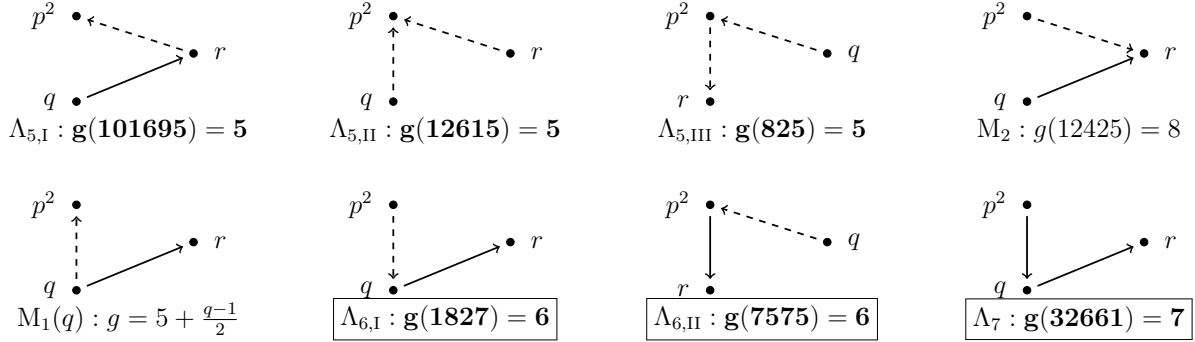
\begin{center}

\resizebox{\textwidth}{!}{%
\begin{tikzpicture}
	

		\begin{scope}[shift={(0, 0)}]
			\node[point, label=left:$p^2$] (p2) {};
			\node[point, label=left:$q$, below=of p2] (q2) {};
			\node[point, label=right:$r$] at ($(q2) + (22.5:2)$)(r2) {};
			\node at ($(r2) + (230:1.6)$) {$\m1(q) : g = 5 + \frac{q - 1}{2}$};
		\end{scope}
		
		\begin{scope}[shift={(5, 0)}]
			\node[point, label=left:$p^2$] (p1) {};
			\node[point, label=left:$q$, below=of p1] (q1) {};
			\node[point, label=right:$r$] at ($(q1) + (22.5:2)$)(r1) {};
			\node at ($(r1) + (230:1.6)$) {$\boxed{\slame: \mathbf{g(1827) = 6}}$};
		\end{scope}
		
		\begin{scope}[shift={(10, 0)}]
			\node[point, label=left:$p^2$] (p7) {};
			\node[point, label=left:$r$, below=of p7] (q7) {};
			\node[point, label=right:$q$] at ($(q7) + (22.5:2)$)(r7) {};
			\node at ($(r7) + (230:1.6)$) {$\boxed{\slamz : \mathbf{g(7575) = 6}}$};
		\end{scope}
		
		\begin{scope}[shift={(15, 0)}]
			\node[point, label=left:$p^2$] (plast) {};
			\node[point, label=left:$q$, below=of plast] (qlast) {};
			\node[point, label=right:$r$] at ($(qlast) + (22.5:2)$) (rlast) {};
			\node at ($(rlast) + (230:1.6)$) {$\boxed{\hlame : \mathbf{g(32661) = 7}}$};
		\end{scope}


		\begin{scope}[shift={(0, 3)}]
			\node[point, label=left:$p^2$] (p3) {};
			\node[point, label=left:$q$, below=of p3] (q3) {};
			\node[point, label=right:$r$] at ($(q3) + (22.5:2)$)(r3) {};
			\node at ($(r3) + (230:1.6)$) {$\qlame : \mathbf{g(101695) = 5}$};
		\end{scope}		
		
		\begin{scope}[shift={(5, 3)}]
			\node[point, label=left:$p^2$] (p6) {};
			\node[point, label=left:$q$, below=of p6] (q6) {};
			\node[point, label=right:$r$] at ($(q6) + (22.5:2)$)(r6) {};
			\node at ($(r6) + (230:1.6)$) {$\qlamz: \mathbf{g(12615) = 5}$};
		\end{scope}
		
		\begin{scope}[shift={(10, 3)}]
			\node[point, label=left:$p^2$] (p4) {};
			\node[point, label=left:$r$, below=of p4] (q4) {};
			\node[point, label=right:$q$] at ($(q4) + (22.5:2)$)(r4) {};
			\node at ($(r4) + (230:1.6)$) {$\qlamd : \mathbf{g(825) = 5}$};
		\end{scope}

		\begin{scope}[shift={(15, 3)}]
			\node[point, label=left:$p^2$] (p5) {};
			\node[point, label=left:$q$, below=of p5] (q5) {};
			\node[point, label=right:$r$] at ($(q5) + (22.5:2)$)(r5) {};
			\node at ($(r5) + (230:1.6)$) {$\m{2}: g(12425) = 8$};
		\end{scope}

		\path[->, thick] (q2) edge[dashed] (p2) (q2) edge (r2)
		(q6) edge[dashed]  (p6) (r6) edge[dashed] (p6) (r4) edge[dashed] (p4) (r7) edge[dashed] (p7);
		\path[->, thick] (q1) edge (r1) (q3) edge (r3) (q5) edge (r5);
		\path[->, thick] (plast) edge (qlast);
		\path[->, thick] (qlast) edge (rlast);
		\path[->, thick] (r3) edge[dashed] (p3) (p5) edge[dashed] (r5);
		\path[->, thick] (p1) edge[dashed] (q1);
		\path[->, thick] (p4) edge[dashed] (q4);
		\path[->, thick] (p7) edge (q7);
\end{tikzpicture}
}
\end{center}
\subsubsection*{Regular \hgs{1} with three edges}
A new edge is to be added to an admissible graph with two edges and $g \le 6$, so the candidates are $\qlame, \qlamz, \qlamd, \slame$ and $\slamz$. Because the graph is acyclic and $r > q$, the direction of the edge is uniquely determined. We now study the effect of adding an edge to these five graphs:
\begin{itemize}
	\item $\qlame$: We get at least $w_{p^2 - 1}(qr) + w_{p + 1}(q) + \frac{q - 1}{2}w_{p + 1}(q)w_{r - 1}(q) \ge 3$ additional groups.
	\item $\qlamz$: We get the same graph from the previous case.
	\item $\slame$: We get $D(p)$.
	\item $\qlamd$ and $\slamz$: We get at least $2w_{r - 1}(pq) + 2w_{r - 1}(q) = 4$ additional groups.
\end{itemize} \nopagebreak[4]
We conclude that we cannot get an admissible graph with three edges.

\subsection{Large non-square-free admissible numbers} 
Finally, it remains to handle non-square-free admissible $n$ with $\lambda(n) = 5$. There are three cases to consider.

\case{A} $n = p^3 q r$.
We immediately see that $q$ and $r$ are not related, since otherwise we would get at least $g(p^3)g(qr) = 10$ groups. We must have $g(p^3 q) \ge 6$, and by connectivity, we have $g(p^3 q) = 6$ and $r \mid p^\alpha - 1$ for some $1 \le \alpha \le 3$. The case that yields the least number of groups is $r \mid p^3 - 1$ with $r \nmid p^2 - 1$, and in this case we get $2^2 = 4$ groups of the form $\cyc{p}^3 \rtimes \cyc{qr}$ and 4 groups of the form $P \times \cyc{qr}$, where $P$ is any group of order $p^3$ different from $\cyc{p}^3$. This gives a total of 8 groups, proving that $n$ is inadmissible.

\case{B} $n = p^2 q^2 r$.
We must have $g(p^2 r) = g(q^2 r) = 3$, and therefore arrows $r \dashrightarrow p^2, q^2$. It is easy to see that there are $3^2 = 9$ groups of the form $(P \times Q) \rtimes \cyc{r}$, simply by observing that for each $P$ we can choose either $\cyc{p^2}\text{ or }\cyc{p}^2$, and in the latter case, we can let $\cyc{r}$ act on it or not. Once again, $n$ is inadmissible.
 
\case{C} $n = p^2 q r s$. It is clear that $qrs$ has at most two edges, because otherwise we have $g(qrs) \ge 4$ by Hölder's formula and thus $g(n) \ge 9$.

\scase{C.1} \emph{$qrs$ has precisely one edge $q \rightarrow r$.} We claim that $n$ is inadmissible.\nopagebreak

For a set of primes $\alpha$, let $g^{(k)}_\alpha(n)$ denote the number of groups of order $n$, with the Sylow subgroups corresponding to at least $k$ primes in $\alpha$ being direct factors (we can omit the superscript if $k = 1$). It is clear that \smash{$g_{p, q}(n) = g_p(n) + g_q(n) - g^{(2)}_{\{p, q\}}(n)$}, by the inclusion-exclusion principle. Moreover, $g_p(n) = g(p^\nu)g(n/p^\nu)$ where $p^\nu$ is the largest power of $p$ diving $n$.

The strategy will be to choose a suitable $\alpha$ and use these facts to show that $g_\alpha(n) \ge 7$ for all possible graphs of $n$, and it will be clear that there is a group which has no direct $p$-Sylow subgroup factor for any $p$ in $\alpha$, showing that $g(n) > g_\alpha(n)$.

We need $g(p^2 s)g(qr) + 1 \le 7$, which forces $g(p^2 s) \le 3$. By connectivity, $g(p^2 s) = 3$ and thus we have a weak arrow $s \dashrightarrow p^2$. At least one edge must be added for the graph of $n$ to be connected. By focusing on $p^2 q r$, we immediately see that there are three options, corresponding to $M_1(q), \slame$ and $\qlame$:\drawunspace
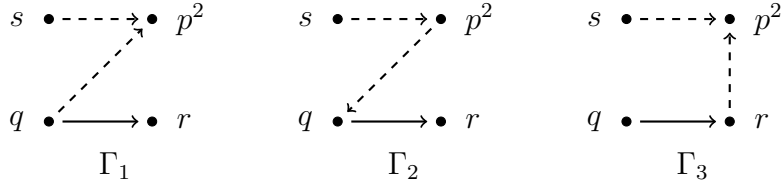
\begin{center}
	\begin{tikzpicture}[scale=0.95]
		\begin{scope}[shift={(0, 0)}]
			\node[point, label=left:$s$] (s1) {};
			\node[point, label=right:$p^2$, right=of s1] (p1) {};
			\node[point, label=left:$q$, below=of s1] (q1) {};
			\node[point, label=right:$r$, below=of p1] (r1) {};
			\node at ($(r1) + (230:0.8)$) {$\Gamma_1$};
		\end{scope}
		
		\begin{scope}[shift={(4, 0)}]
			\node[point, label=left:$s$] (s2) {};
			\node[point, label=right:$p^2$, right=of s2] (p2) {};
			\node[point, label=left:$q$, below=of s2] (q2) {};
			\node[point, label=right:$r$, below=of p2]	(r2) {};
			\node at ($(r2) + (230:0.8)$) {$\Gamma_2$};		
		\end{scope}
		
		\begin{scope}[shift={(8, 0)}]
			\node[point, label=left:$s$] (s3) {};
			\node[point, label=right:$p^2$, right=of s3] (p3) {};
			\node[point, label=left:$q$, below=of s3] (q3) {};
			\node[point, label=right:$r$, below=of p3] (r3) {};
			\node at ($(r3) + (230:0.8)$) {$\Gamma_3$};
		\end{scope}

		\path[->, thick] (q1) edge (r1) (q2) edge (r2) (q3) edge (r3);
		\path[->, thick] (s1) edge[dashed] (p1) (s2) edge[dashed] (p2) (s3) edge[dashed] (p3);
		\path[->, thick] (q1) edge[dashed] (p1) (p2) edge[dashed] (q2);
		\path[->, thick] (r3) edge[dashed] (p3);
	\end{tikzpicture}
\end{center}\drawunspace
All that remains is to consider suitable subsets $\alpha$ and compute $g(\Gamma)$ by our earlier analysis of regular graphs:
\begin{itemize}
	\item $\Gamma_1$: We have $g_{\{r, s\}}(p^2 q r s) = g(p^2 q s) + g(p^2 q r) - g(p^2 q) = 5 + 5 + \frac{q - 1}{2} - 3 \ge 8.$
	\item $\Gamma_2$: We have $g_{\{r, s\}}(p^2 q r s) = g(p^2 q s) + g(p^2 q r) - g(p^2 q) = 5 + 6 - 4 = 7.$
	\item $\Gamma_3$: We have $g_{\{q, s\}}(p^2 q r s) = g(p^2 r s) + g(p^2 q r) - g(p^2 r) = 5 + 5 - 3 = 7.$
\end{itemize}
We conclude that $g(n) > 7$.

\scase{C.2} \emph{qrs has precisely two edges}. We claim that $n$ is inadmissible as well.

The graph of $qrs$ is \tikz[ww] \pic {twalk};. At least one prime, say $v$, is connected with $p^2$. We need $g(p^2 v) = 3$, which implies the existence of an arrow $v \dashrightarrow p^2$ by \hthref{euppq}.

Without loss of generality, we assume that $q < r < s$. We cannot have an arrow $r \dashrightarrow p^2$ because then the graph of $p^2 r s$ will be an $M_1(r)$ with $r > 3$. If we have an arrow $q \dashrightarrow p^2$, we calculate $g_{\{p, s\}}(n) = 2g(qrs) + g(p^2 q r) - 2g(qr) = 6 + 6 - 4 = 8.$ If instead we have $s \dashrightarrow p^2$, we have $g_{\{q, p\}}(n) = g(p^2 r s) + 2g(qrs) - 2g(rs) = 7$. By the same reasoning, we conclude that $g(n) \ge 8$, completing the proof.

\scase{C.3} \emph{$qrs$ is cyclic.} It is clear that the out-degree of $p^2$ is $1$ and the in-degree is $2$. The arrows to $p^2$ are certainly weak, so we let $q < r$ with $q, r \dashrightarrow p^s$. This means we have an arrow from $p^2$ to $s$. We show that, when this arrow is weak, we get $g(n) = 7$. Replacing it with a strong arrow increases $g$, making $n$ inadmissible.

Let $G$ be a group of order $n = p^2 q r s$. We claim that either $G \cong G_0 \times \cyc{q}$ or $G \cong G_1\times \cyc{s}$. This will prove the result, since it implies $g(n) = g_{\{q, s\}}(n)$. By the same reasoning of the previous sub-case, we see that this is equal to $g(p^2 r s) + g(p^2 q r) - g(p^2 r) = g(\qlamd) + g(\qlamz) - g(p^2 r) = 7$.

It is clear that $n$ is odd, so $G$ is solvable by the Feit-Thompson theorem {\cite{oddsolve}}. Consequently, we have a subgroup $G_0$ of index $q$ by Hall's theorems {\cite[Thm.~3.13]{fgt}}, and since $q$ is the smallest prime dividing $\abs{G}$, it follows from a standard theorem {\cite[Prob.~1A.1]{fgt}} that $G_0$ is normal. Thus if a $q$-Sylow subgroup is normal, we get the first half of the claim. Suppose then that there is no normal $q$-Sylow subgroup.

Consider a Hall subgroup $C$ of order $qrs$, which is cyclic by definition. Since the $s$- and $q$-Sylow subgroups are normalized by $C$, we have $n_s(G), n_q(G) \mid p^2$. We have $n_q(G) = p^2$, since $n_q(G) = p$ implies the existence of a strong arrow \nopagebreak[4] $q \rightarrow p$, contrary to our assumption. Next, we cannot have $n_s(G) = p$ because $s > p$, and we cannot have $n_s(G) = p^2$ because otherwise $p^2 \equiv 1 \Mod{s}$ and thus $p \equiv s - 1 \Mod{s}$. Since $s > p$, this implies $p = s - 1$, which is impossible since $(s, p) \neq (3, 2)$. We conclude that the $s$-Sylow subgroup is normal. 

To see that its complement is normal as well, let $S, P, Q,\text{ and }R$ denote Sylow subgroups corresponding to the primes $s, p, q,\text{ and }r$ respectively. We know that $H = S(PQ)$ is a subgroup, and $S$ is normal in it, so we have $H \cong \cyc{s} \rtimes H_0$ where $H_0$ is a group of \mbox{order $p^2 q$.}

Consider the homomorphism $\phi : H_0 \rightarrow \aut{\cyc{s}} \cong \cyc{s - 1}$ defining the semidirect product. We know that $Q \subseteq \ker \phi$ since $q \nmid s - 1$. If $Q = \ker \phi$, $Q$ will be normalized by $P$, which implies $n_q(G) = 1$, contrary to our assumption. Similarly, we cannot have $\abs{\ker \phi} = pq$ because $pq$ is a cyclic number, so we would get $n_q(G) = p$, which has already been ruled out. It follows that $\phi$ must be the trivial homomorphism, and we conclude that $S$ normalizes $PQ$. It normalizes $R$ for similar reasons since $sr$ is a cyclic number.  It follows that the complement is normal, and this proves the second half of the claim.
\section{Summary of the results}
In what follows, $s, p, q, r$ represent distinct, odd primes.
\nopagebreak
\setenumerate[1]{label=\Roman*., font=\normalfont}
\begin{thm}
	For cyclic-free $n$, we have $g(n) = 6$ precisely when one of the \mbox{following holds:}
	\begin{enumerate} \listspace
		\item$n = pqrs$ where $r \equiv 1 \Mod{qs}$ and $q \equiv 1 \Mod{p}$,\nopagebreak[4] 
		\item\nopagebreak[4] $n = 2pq$ where $q \equiv 1 \Mod{p}$,
		\item $n = 3p^2$ where $p \equiv 1 \Mod{3}$,
		\item $n = p^3 q$ where $p^3 \equiv 1 \Mod{q}$,
		\item $n = p^2 q r$, and one of the following conditions are satisfied:
		\begin{enumerate}
			\item $p \parallel q - 1$ and $r \equiv 1 \Mod{q}$,
			\item $q = 3$ and $p \equiv -1 \Mod{3}$, \nopagebreak[4]
			\item\nopagebreak[4] $p^2 \mid r - 1$ and $p \equiv -1 \Mod{q}$,
		\end{enumerate}
		\item $n = p^2 q^2$ where $p \parallel q + 1$,
		\item $n = n_1 n_2$ where $n_1, n_2$ are arithmetically independent and
		\begin{enumerate}
			\item $n_1 = qr$ where $r \equiv 1 \Mod{q}$ or $n_1 = q^2$,
			\item $n_2 = pqr$ where $r \equiv 1 \Mod{q}$ and $q \equiv 1 \Mod{p}$ or $n_2 = p^2 q$ and $q \mid p + 1$,
		\end{enumerate}
	\end{enumerate} \textspace
	\nopagebreak[4] where no further congruences of the form $\alpha \equiv \pm 1 \Mod{\beta}$ \nolinebreak[4] \mbox{occur}.
\end{thm}
\begin{keepintact}
\setenumerate[1]{label=\Roman*., font=\normalfont}
\begin{thm}
	For cyclic-free $n$, we have $g(n) = 7$ precisely when one of the \mbox{following holds:}
	\begin{enumerate}	\listspace
		\item $n = 5pq$ where $p \equiv q \equiv 1 \Mod{5}$,
		\item $n = 3qr$ where $q \equiv r \equiv 1 \Mod{3}$ and $r \equiv 1 \Mod{q}$,
		\item $n = 3pqr$ where $p \equiv q \equiv 1 \Mod{3}$ and $r \equiv 1 \Mod{q}$,
		\item $n = 5p^2$ where $p \equiv 1 \Mod{5}$,
		\item $n = p^3 q$ where $p \equiv -1 \Mod{q}$,
		\item $n = 5p^2 q$ where $q \equiv 1 \Mod{5}$ and $p \equiv -1 \Mod{5}$,
		\item $n = p^2 q r$ where $p^2 \mid q - 1$ and $r \equiv 1 \Mod{q}$,
		\item $n = p^2 q^2$ where $p^2 \mid q + 1$,
		\item $n = p^2 q r s$ where $p \equiv -1 \Mod{qr}$ and $p \parallel s - 1$,
	\end{enumerate} \textspace
	where no other congruences of the form $\alpha \equiv \pm 1 \Mod{\beta}$ \nolinebreak[4] \mbox{occur}.
\end{thm}
\end{keepintact}

\begin{center}\section*{References}\end{center}
\microtypesetup{protrusion=false}
\printbibliography[heading=none]

\end{document}